\documentclass[11pt]{article}

\usepackage{amsthm, amsmath,amssymb, amsfonts, fancyhdr}

\usepackage{graphicx}

\usepackage{setspace}

\theoremstyle{plain}
\newtheorem{thm}{Theorem}[section]
\newtheorem{lem}[thm]{Lemma}
\newtheorem{prop}[thm]{Proposition}
\newtheorem{cor}[thm]{Corollary}
\newtheorem{conj}[thm]{Conjecture}
\newtheorem{princ}[thm]{Principle}

\theoremstyle{definition}
\newtheorem{defn}{Definition}[section]

\begin{document}

\title{Freiman-Ruzsa-Type Theory For Small Doubling Constant}
\author{Hansheng Diao}

\maketitle
\begin{abstract}
In this paper, we study the linear structure of sets $A \subset
\mathbb{F}_2^n$ with doubling constant $\sigma(A)<2$, where
$\sigma(A):=\frac{|A+A|}{|A|}$.  In particular, we show that $A$ is
contained in a small affine subspace. We also show that $A$ can be
covered by at most four shifts of some subspace $V$ with $|V|\leq
|A|$. Finally, we classify all binary sets with small doubling
constant.
\end{abstract}

\section{Introduction and Statement of Results}
Let $A$ be a finite subset of an abelian group $G$, and define the
\emph{doubling constant} $\sigma (A)$ of  $A$ to be
 \[\sigma(A):= \frac{|A+A|}{|A|},\]where $A+A$ is the collection
of all sums $a+a'$ with $a,a'\in A$.

Suppose that $A$ is nearly closed under group addition in the sense that
$\sigma(A) \leq K$ for some small $K$, what can be said about the structure of $A$? It
is easy to see that such sets will  possess good linear structures. In
particular, they will be cosets of subspaces or large subsets of
them. In the 1970s, Freiman obtained the following celebrated
theorem for $G=\mathbb{Z}$~\cite{F1}:
\begin{thm}[Freiman]
If $A\subseteq \mathbb{Z}$ and if $|A+A| \leq K|A|$, then $A$ is
contained in a proper arithmetic progression of dimension $d$ and
size at most $s|A|$ such that $d$ and $s$ only depend on $K$.
\end{thm}

In general, B. Green and I. Z. Ruzsa proved similar results for
arbitrary abelian groups~\cite{G2}. The precise statement
in~\cite{G2} is the following:
\begin{thm}[B. Green and I. Z. Ruzsa]
Let $A\subseteq G$ satisfy $|A+A| \leq K|A|$. Then $A$ is contained
in a coset progression of dimension $d(K)$ and size at most
$f(K)|A|$. One may take $d(K)=CK^4\textrm{log }^2(K+2)$ and
$f(K)=\textrm{exp }(CK^4\textrm{log}^2(K+2))$ for some absolute
constant $C$.
\end{thm}
 As we see in the statement of this theorem, the analogue of arithmetic
progression in an arbitrary abelian group is \emph{coset
progression}. By a \emph{coset progression} of dimension $d$ we mean
a subset of $G$ of the form $P+H$, where $H$ is a subgroup of $G$,
and $P$ is a proper progression of dimension $d$.

The theorem above gives the best bound known for an arbitrary group.
However, better estimates can be derived for specific abelian groups. We consider the case where $G=\mathbb{F}_2^n$, which is particularly interesting because of its usefulness for theoretical computer science. See~\cite{C,F2,Z} for more information.\\

To study the structure of small doubling subset of $\mathbb{F}_2^n$,
B. Green and T. Tao introduced the following definition of $F(K)$
and $G(K)$~\cite{G1}.
\begin{defn}
Define $F(K)$ to be the least positive constant such that for any $m
\in \mathbb{Z}^+$ and any non-empty set $A \subseteq \mathbb{F}^m_2$
with $\sigma(A) \leq K$, $A$ is contained in an affine subspace $V
\subseteq \mathbb{F}^m_2$ of cardinality $|V| \leq F(K) |A|$.
\end{defn}

\begin{defn}
Define $G(K)$ to be the least positive integer such that for any $m
\in \mathbb{Z}^+$ and any non-empty set $A \subseteq \mathbb{F}^m_2$
with $\sigma(A) \leq K$, there exists a linear subspace $V \subseteq
\mathbb{F}_2^m$ of cardinality $|V| \leq |A|$ such that $A$ is
covered by at most $G(K)$ translates of $V$.
\end{defn}

In the past five years, the bounds for the value of $F(K)$ and
$G(K)$ have been continually improved. Some of the best results so
far are listed below.

\begin{thm}[B. Green and T. Tao~\cite{G1}]
\[F(K)=2^{2K+O(\sqrt{K}\textrm{log }K)}.\]
\end{thm}

\begin{thm}[{A quick result follows from [7, Corollary 1.5] and Ruzsa's covering lemma [10, Lemma
2.14]}]
\[G(K)\ll K^{O(K)}.\]
\end{thm}

\begin{thm}[Deshouillers, Hennecart and Plagne~\cite{D}]
For $1\leq K <4$, \[F(K) \leq \frac{2K-1}{3K-1-K^2}.\]
\end{thm}

In~\cite{G1}, the authors give the following formulae for $F(K)$ and
$G(K)$ in the region $1 \leq K < \frac{9}{5}$:
\begin{prop}[B. Green and T. Tao~\cite{G1}]
\[F(K)=\left\{\begin{array}{ll}
K & 1 \leq K < \frac{7}{4};\\
\frac{8}{7}K &  \frac{7}{4} \leq K < \frac{9}{5}.\end{array}\right.
\]
\end{prop}

\begin{prop}[B. Green and T. Tao~\cite{G1}]
\[G(K)=\left\{\begin{array}{ll}
2 & 1 < K < \frac{7}{4};\\
3 & \frac{7}{4} \leq K < \frac{9}{5}.\end{array}\right.
\]
\end{prop}

In this paper, we will extend these results by determining the exact
value of $F(K)$ and $G(K)$ in the region $1 \leq K < 2$, and thus
classify all small doubling sets. In particular, we will prove the
following:

\begin{thm}

\[F(K)=\left\{\begin{array}{ll}
K & 1 \leq K < \frac{7}{4};\\
\frac{8}{7}K &  \frac{7}{4} \leq K < 2.\end{array}\right.
\]
\end{thm}
\begin{thm}
\[G(K)=\left\{\begin{array}{ll}
2 &  1 < K < \frac{7}{4};\\
3 & \frac{7}{4} \leq K <
\frac{31}{16};\\
4 & \frac{31}{16} \leq K < 2.\end{array}\right.
\]

\end{thm}

To this end, we need to study the structure of \emph{normal sets},
i.e., those sets $T \subseteq \mathbb{F}_2^m$ with $|T+T|=2|T|-1$.
The definition will be carefully introduced in section 2. Two other
key definitions, \emph{normal numbers} and \emph{rank of a normal
set}, will also be given in the same section.

In section 3, we give a formula of $F(K)$ in terms of normal numbers
for $\frac{7}{4} \leq K < 2$. By a similar method, we derive an
estimate for $G(K)$ in section 4.

In section 5, we determine all normal numbers and their
corresponding ranks via Kemperman-type theory, thus determining the
exact value of $F(K)$. Being more careful, I prove some better
bounds for $G(K)$ in section 6, which will determine the exact value
of $G(K)$ for $1 < K < 2$. This will complete the proof of Theorem
1.8 and Theorem 1.9.

In the last section, we will discuss how the small doubling subsets
of $\mathbb{F}_2^n$ are intrinsically related to the normal sets,
and then completely classify the small doubling sets with
$\sigma(A)<2$.

\section{Definitions}

\begin{defn}
(Normal numbers) A positive integer $n \geq 4$ is called a
\emph{normal number} if there exists an $m \in \mathbb{Z}^+$ and a
set $T \subseteq \mathbb{F}^m_2$ such that $|T| = n$ and $|T+T| =
2n-1$. Such a set $T$ is called a \emph{normal set of level $n$}.
\end{defn}

There are infinitely many normal numbers. For instance, consider
$T=\{(a_1,a_2,...,a_m)\in \mathbb{F}_2^{n+1}\,\,|\,\,a_1=0,
a_2,a_3,...,a_m \textrm{ not all }1's \}\cup \{(1,0,0,...,0)\}$ then
$|T|=2^n$ and $|T+T|=|\mathbb{F}_2^{n+1}-
\{(1,1,1,...,1)\}|=2^{n+1}-1$. This means $2^n$ ($n\geq 2$) is a
normal number. On the other hand, not all numbers are normal. For
example, 5 is not normal.

Let $S$ denote the set of all normal numbers, and $\Sigma_n$ denote
the set of all normal sets of level $n$. Write $S =
\{n_1<n_2<n_3<...\}$. $n_i$ is called \emph{the i-th normal number}.
In particular, $n_1=4$.

Let us next recall the definition of \emph{Freiman
$s$-isomorphism}.
\begin{defn}
(Freiman isomorphism) Let $s \geq 2$ be an integer. Let $G, G'$ be
two abelian groups and let $A\subseteq G$ and $A' \subseteq G'$ be
subsets. A map $\phi : A \rightarrow A'$ is called a \emph{Freiman
$s$-homomorphism} if whenever $a_1,...,a_s,b_1,...,b_s \in A$
satisfy \[ a_1+a_2+...+a_s = b_1+b_2+...+b_s\] we have
\[\phi(a_1)+\phi(a_2)+...+\phi(a_s) =
\phi(b_1)+\phi(b_2)+...+\phi(b_s).\] If $\phi^{-1}$ is also a
Freiman s-homomorphism, then we say that $\phi$ is a \emph{Freiman
$s$-isomorphism}, and write $A\cong_s A'$.
\end{defn}
Now suppose $T$ is a normal set of level $n$ and $T'$, a subset of
$\mathbb{F}^m_2$, is Freiman 2-isomorphic to $T$. Then $|T'|=|T|=n$,
$|T'+T'|=|T+T|=2n-1$. Hence, $T'$ is also a normal set of level $n$.
This motivates the following definition.

\begin{defn}
(rank of normal set) Let $T$ be a normal set of level $n$. The
\emph{rank} of $T$ is the least integer $m$ such that there exists a
set $T' \subseteq \mathbb{F}_2^m$ which is Freiman 2-isomorphic to
$T$; i.e. $T' \cong_2 T$. Denote the rank of $T$ by $R(n, T)$.
\end{defn}

\begin{lem}
For any $n \in S$ and $T \in \Sigma_n$, we have $R(n, T) \leq n$.
\end{lem}

\begin{proof}
Let $v_1, v_2, \cdots, v_m$ ($m\leq n$) be a maximal collection of
linear independent elements in $T$. Then the map $\phi: v_i\mapsto
e_i\in \mathbb{F}_2^m$ induces an Freiman 2-isomorphism from $T$ to
$\phi(T)\subset \mathbb{F}_2^m$.
\end{proof}

\begin{defn}
For $n \in S$, define $C(n) = \sup_{T \in S_n} R(n, T)$.
\end{defn}

The existence of $C(n)$ is guaranteed by Lemma 2.1.

\section{A Formula for $F(K)$}
\begin{prop}
Let $n_i$ denote the $i$-th normal number. Then
\[F(K) = \max_{1 \leq j\leq i}
\frac{2^{C(n_j)}}{2n_j-1}K
\] if $\frac{2n_i-1}{n_i}\leq K < \frac{2n_{i+1}-1}{n_{i+1}}$.
\end{prop}

In this section, I will prove Proposition 3.1 by showing $F(K) \geq
\max_{1 \leq j\leq i} \frac{2^{C(n_j)}}{2n_j-1}K$ and $F(K) \leq
\max_{1 \leq j\leq i} \frac{2^{C(n_j)}}{2n_j-1}K$, respectively.

\subsection{The Lower Bound}
It suffices to prove that \[F(K)\geq \frac{2^{C(n_j)}}{2n_j-1}K.\]
when $K \geq \frac{2n_j-1}{n_j}$\\ By definition, there exists a
normal set $T \subseteq \mathbb{F}_2^{C(n_j)}$ of level $n_j$ such
that $C(n_j)=R(n_j,T)$. It follows that $T$ is not contained in any
non-trivial affine subspace of $\mathbb{F}_2^{C(n_j)}$. Let $A$ be a
random subset\footnote{By a \emph{random subset} $A$ we mean a set
without too much linear structure such that $A+A$ covers the whole
set $(T+T)\times \mathbb{F}_2^{n-C(n_j)}$.} of $T \times
\mathbb{F}_2^{n-C(n_j)}$ with $n$ large of density close to
$\frac{2n_j-1}{n_j K}$. Then the smallest affine subspace of
$\mathbb{F}_2^n$ containing $A$ is $\mathbb{F}_2^n$ itself.\\
Therefore, \[F(K) \geq \frac{2^n}{\frac{2n_j-1}{n_j K} \times n_j
\cdot 2^{n-C(n_j)}} = \frac{2^{C(n_j)}}{2n_j-1}K\]

\subsection{The Upper Bound}
Here I follow B. Green and T. Tao's idea in~\cite{G1}.\\
Suppose $K < \frac{2n_{i+1}-1}{n_{i+1}}$. Let $H$ be the largest
subspace of $\mathbb{F}_2^n$ such that $A+A$ is a union of cosets of
$H$. From Kneser's theorem (see [10, Theorem 5.5]) we have $|A+A|
\geq 2|A|-|H|$. Since $K < \frac{2n_{i+1}-1}{n_{i+1}}$, the
inequality gives $|A+A| < (2n_{i+1}-1)|H|$. If we write $B:=(A+H)/H
\subseteq \mathbb{F}_2^n/H$, then $B+B$ has cardinality at most
$2n_{i+1}-2$. Moreover, $B+B$ cannot be expressed as the union of
cosets of any non-trivial subspace of $\mathbb{F}_2^n/H$. Kneser's
theorem gives
$|B+B| \geq 2|B| - 1$ which implies $|B| \leq n_{i+1}-1$. \\
On the other hand, it is clear that $\frac{|B+B|}{|B|}\leq K <2$.
Together with Kneser's theorem, we get $|B+B| = 2|B|-1$. Hence $B$
is Freiman 2-isomorphic to a normal set. In particular, $|B| = n_j$
for some $j \in \{1, 2, ..., i\}$.\\
It follows that $A$ can be covered by the a subspace isomorphic to
$\mathbb{F}_2^{C(n_j)} \times H$ with cardinality $2^{C(n_j)}|H|\leq
\frac{2^{C(n_j)}}{2n_j-1}K \cdot |A|$. Therefore,
\[F(K)\leq \max_{1 \leq j\leq i} \frac{2^{C(n_j)}}{2n_j-1}K.\]

\section{An Estimate for $G(K)$}
Instead of a precise formula, I give the following estimate for
$G(K)$.
\begin{prop}
Let $n_i$ denote the $i$-th normal number. Then
\[ \max_{1\leq j \leq i}(C(n_j)+2-\lceil\log_2n_j\rceil) \leq G(K) \leq \max_{1\leq j\leq i}2^{C(n_j)-\lfloor\log_2(n_j-1)\rfloor} \]
if $\frac{2n_i-1}{n_i}\leq K < \frac{2n_{i+1}-1}{n_{i+1}}$.
\end{prop}

\subsection{The Lower Bound}
Suppose $\frac{2n_j-1}{n_j} \leq K <2$, we show that
\[G(K) \geq C(n_j)+2-\lceil\log_2n_j\rceil.\]
Consider a normal set $T \subseteq \mathbb{F}_2^{C(n_j)}$ of level
$n_j$ satisfying $R(n_j, T)=C(n_j)$. Write $m =C(n_j)$. Under a
proper linear transformation, we may assume that
$\{0,e_1,e_2,...,e_m\} \subseteq T$. Let $A$ be the Cartesian
product (in $\mathbb{F}_2^m \times \mathbb{F}_2^n =
\mathbb{F}_2^{m+n}$) of $T$ together with a random subset $A'$ in
$\mathbb{F}_2^n$ of density close to $\frac{2n_j-1}{n_j K}$. (When choosing $A'$, let $\{0, e_{m+1}, e_{m+2}, ...,e_{m+n}\} \subseteq A'$.)\\
Suppose $A$ is covered by the union of $a_1+V, a_2+V, ..., a_l+V$
where $V \subseteq \mathbb{F}_2^{m+n}$ is a linear subspace of
cardinality $|V| \leq |A|$. Then \[\textrm{dim }V \leq
\log_2\lfloor|A|\rfloor \leq n-1+\lceil\log_2n_j\rceil.\] Assume
that $a_1=0$. Let
$M_t:= (a_t+V)\cap \{e_1, e_2, ..., e_{m+n}\}$, $ t=1,2,...,l$.\\
Take a representative element $e(t)$ from each set $M_t$ (If
$M_t=\phi$, simply set $e(t)=0$). Since $V$ is a linear space, the
difference of two elements in the same $M_t$ belongs to $V$. Hence
\[M_1 \cup (\bigcup_{t=2}^l \{e - e(t) | e \in M_t\backslash e(t)\}
) \subseteq V.\] However all the elements in this set are linearly
independent. So
\[\textrm{dim }V \geq |M_1| + \sum_{t=2}^l |\{e - e(t) | e \in
M_t\}| \geq |M_1|+\sum_{t=2}^l (|M_t|-1) = m+n-l+1.\] Therefore,
\[l\geq C(n_j) +2 - \lceil\log_2n_j\rceil.\]

\subsection{The Upper Bound}
Suppose $K < \frac{2n_{i+1}-1}{n_{i+1}}$. Let $H$ and $B$ be the same as in \S 3.2.\\
Then \[|H| \leq \frac{K}{2n_j-1}|A| < \frac{1}{n_j-1}|A|.\] Note
that $A$ is contained in some linear space $F$
isomorphic to $\mathbb{F}_2^{C(n_j)}\times H $ and $\dim F = C(n_j)+\dim H$.\\
Let $l = \lfloor \log_2(n_j-1)\rfloor$ and let $H'$ be a linear
subspace of $F$ containing $H$ with $\dim H' = \dim H+ l$. Then
$|H'|<|A|$ and $F$ is the union of
\[2^{\dim F-l-\dim H}= 2^{m-l}=2^{C(n_j)-\lfloor\log_2(n_j-1)\rfloor}\] cosets of $H'$. Therefore,
\[G(K)\leq \max_{1\leq j\leq i}2^{P(n_j)-\lfloor\log_2(n_j-1)\rfloor} .\]

\section{Structure of Normal Sets}
In~\cite{K}, Kemperman describes the structure of subsets $A$, $B$
of an abelian group $G$ satisfying $|A+B|= |A|+|B|-1$. In
particular, if we set $A = B$ and $G = \mathbb{F}_2^n$, then
Kemperman's theorem gives the structure of normal sets.

In the language of~\cite{L}, V. Lev proved the following special
case for $G = \mathbb{F}_2^r$:
\begin{thm}[V. Lev~\cite{L}]
Let $r>1$ be an integer. If a subset $A \subseteq \mathbb{F}_2^r$
satisfies $|A+A|<2|A|$, then one of the following holds:\\
(i) there exists a subgroup $H \leq \mathbb{F}_2^r$ such that $A$ is
contained in an $H$-coset and $|A|>|H|/2$; \\
(ii)there exists two subgroups $F, H \leq \mathbb{F}_2^r$,
satisfying $|F|\leq 8$ and $F\cap H = \{0\}$, and an aperiodic
antisymmetric subset $F_1 \subseteq F$, such that $A$ is obtained
from a shift of the set $F_1+H$ by removing less than $|H|/2$ of its
elements. In this case $A+A$ is the sum $F\oplus H$ with one
$H$-coset removed, so that $|A+A|=(|F|-1)|H|$.

\end{thm}

Based on V. Lev's result, we can determine all normal numbers and
the corresponding ranks.
\begin{thm}
Let $n_i$ denote the $i$-th normal number. Then
\begin{itemize}
\item $n_i=2^{i+1} $;
\item $C(n_i)=i+2 $.
\end{itemize}
\end{thm}

\begin{proof}
Let $T \subseteq \mathbb{F}_2^{R(n,T)}$ be a normal set of level
$n$.
Then $T$ satisfies the condition in Theorem 5.1.\\
If (i) holds, then $H$ must be $\mathbb{F}_2^{R(n,T)}$ itself.
However, in this case, $|T+T|=2|T|-1 >2^{R(n,T)}$,
a contradiction. \\
If (ii) holds, then $|T+T|=(|F|-1)|H|$. Note that $|H|$ is a power
of 2 and $|T+T|$ is an odd number. So, $|H|=1$. It follows that $T$
is contained in a shift of $F$ which must be the whole space
$\mathbb{F}_2^{R(n,T)}$. So $2n-1=|T+T|=|F|-1=2^{R(n,T)}-1$.
Thus, $n$ is a power of 2. \\
Recall that $2^n$ ($n\geq 2$) are normal numbers. We obtain $n_i=2^{i+1}$.\\
Furthermore, $R(n_i,T)=1+\log_2n_i=i+2$, which is independent of $T$.\\
Therefore, $C(n_i)=i+2$.
\end{proof}

\begin{proof}[Proof of Theorem 1.8]
 By Theorem 3.1, it suffices to compute
$\frac{2^{C(n_j)}}{2n_j-1}$. In fact,
\[\frac{2^{C(n_j)}}{2n_j-1}=\frac{2^{i+2}}{2^{i+2}-1} \leq \frac{2^3}{2^3-1}=\frac{8}{7}\]which implies that $F(K)=\frac{8}{7}K$ for $\frac{7}{4}\leq K<2$.
\end{proof}

\section{The Exact Value of $G(K)$}
For any $j\geq 1$, we have that $C(n_j)+2-\lceil \log_2 n_j
\rceil=j+2+2-(j-1)=3$, and that $C(n_j)-\lfloor\log_2(n_j-1)\rfloor
= (j+2)-j=2$. By Proposition 4.1, we have $3 \leq G(K)\leq 4$ for
$\frac{7}{4}\leq
K<2$.\\

\subsection{When $\frac{7}{4} \leq K < \frac{31}{16}$}

\begin{prop}
When $\frac{7}{4}\leq K<\frac{31}{16}$, $G(K)=3$.
\end{prop}
\begin{proof} Let $H$ and $B$ be the same as in \S
3.2. In this case, we have $|B|\leq 8$.\\
If $|B|\leq 4$, then $A$ is contained in the union of no more than
four cosets of some subspace $H$, where $|H|<\frac{|A|}{2}$.
Covering these four cosets of $H$ by three cosets of a subspace of
one
dimension higher than $H$ we obtain $G(K)\leq 3$.\\
If $|B|=8$, under a proper linear transformation, we can assume that
$\{0, e_1, e_2, e_3, e_4\} \subseteq B$. Let $u$ denote the single
element of $\mathbb{F}_2^4 \backslash (B+B)$. Note that $0, e_i,
e_j+e_k \in B+B$ ($1\leq i,j,k \leq 4$). So $u$ contains three or
four 1's, Without loss of generality, we only consider $u=(1110)$ or
$(1111)$. In either case, it is straightforward to check that $B$
can be covered by three shifts of a linear space with four elements.
If follows that $A$ can always be covered by three cosets of a
linear subspace with cardinality no more than $|A|$.
\end{proof}

\subsection{When $\frac{31}{16}\leq K <2$}
\begin{prop}
When $\frac{31}{16}\leq K<2$, $G(K)=4$.
\end{prop}
\begin{proof}
It suffices to prove that $G(K)\geq 4$.\\
Let $T \subseteq \mathbb{F}_2^5$ be the set of all 5-tuples with no
more than two 1's. Then $T$ is a normal set of level 5. Let $A$ be
the Cartesian product (in $\mathbb{F}_2^n \times \mathbb{F}_2^5 =
\mathbb{F}_2^{n+5}$) of a random subset $A'$ in $\mathbb{F}_2^n$ of
density close to $\frac{31}{16K}$  together with the set
$T$. (When choosing $A$, let $\{0,e_1,...,e_n\} \in A'$)\\
Now suppose $A$ is covered by the union of $V$, $v_1+V$ and $v_2+V$
where $V$ is a linear subspace with $|V| \leq |A|$.  Then there
exists $i,j\in \{1,2,...,n+5\}$ such that $e_i \in v_1+V$ and $e_j
\in v_2+V$ (Otherwise, $V$ contains $n+4$ linearly independent
elements). However, $e_i+e_j$ is contained in $v_1+v_2+V$ which is
disjoint from $A$, a contradiction.
\end{proof}

This completes the proof of Theorem 1.9.

\section{A Complete Classification of Small Doubling sets}

Following the discussion in \S 3 and \S 5, we have
\begin{cor}
For any normal set $T$, there exists an $m \in \mathbb{Z}^+$ and a
normal set $T' \subseteq \mathbb{F}_2^m$ such that
\begin{itemize}
\item $T' \cong_2 T$;
\item $|T|=2^{m-1}$;
\item $\{ 0, e_1, e_2, ..., e_m \} \subseteq T'$.
\end{itemize}
\end{cor}

We call such sets $T'$ \emph{elementary normal sets}.\\
Combining Corollary 7.1 and the discussion in \S 3.2, we can
describe the structure of small doubling sets in the following way:
\begin{prop}
Suppose $A$ is a subset of $\mathbb{F}_2^n$ with $\sigma(A)<2$. Then
there exists an integer $m \leq \lfloor
\log_2(\frac{2}{2-\sigma(A)})\rfloor$, an elementary normal set $T$
of level $2^{m-1}$ and another integer $k \in \mathbb{Z}^+$ such
that $A$ is Freiman 2-isomorphic to a random subset of $T \times
\mathbb{F}_2^k$ with density $\frac{2^m-1}{2^m\sigma(A)}$.
\end{prop}

\begin{proof}
Since $B$ is a normal set, it is 2-isomorphic to an elementary
normal set $T$ of level $2^m$. \\
$H$ is a subspace of $\mathbb{F}_2^n$. So $H$ is isomorphic to
$\mathbb{F}_2^k$ for some $k$. It is easy to compute that the
density of the random subset in $T\times \mathbb{F}_2^k$ is
$\frac{2^m-1}{2^m\sigma(A)}$. Since densities cannot exceed 1, we
have $\frac{2^m-1}{2^m\sigma(A)} \leq 1$. Thus, $m \leq \lfloor
\log_2(\frac{2}{2-\sigma(A)})\rfloor$.
\end{proof}

Therefore, we intrinsically classify all finite binary set with
doubling constant $\sigma(A)<2$. Furthermore, we conclude that:
\begin{princ}
When studying the linear structure of small doubling binary sets in
the sense of $\sigma(A)<2$, it suffices to consider elementary
normal sets.
\end{princ}

\textit{Remark}: The discussion above gives us a general idea for
studying the exact value of $F(K)$ and $G(K)$ when $K\geq 2$: we can
try to give definitions of \emph{generalized normal sets} and study
the linear structure of them. At least, the following conjecture is
reasonable.

\begin{conj}For $K\geq 1$,
\begin{itemize}
\item $F(K)$ is a piecewise linear function;
\item $G(K)$ is a piecewise constant function.
\end{itemize}
\end{conj}

\section*{Acknowledgements}
The author would like to thank professor Ben Green and Catherine
Lennon for many useful suggestions. The work of this paper was done
during the Summer Program of Undergraduate Research (SPUR) at MIT in
July 2007.

\bigskip

\footnotesize\textsc{Department of Mathematics, Massachusetts
Institute of Technology, Cambridge, MA 02139}

\end{document}